\documentclass[14pt]{amsart}
\usepackage[utf8]{inputenc}
\usepackage{amssymb}
\usepackage{amsmath}
\usepackage{hyperref}
\usepackage{tikz}
\usetikzlibrary{shapes,arrows}
\usepackage[english]{babel}
\usepackage{textcase}
\usepackage{comment}
\usepackage{graphicx}
\usepackage[right=3cm, left=3cm]{geometry}

\usepackage{comment}
\usepackage[all,cmtip]{xy}
\usepackage{textcomp}
\usepackage[symbol]{footmisc}
\begin{document}
	
	\setcounter{tocdepth}{0}
	\newtheorem{theorem}{Theorem}[section]
	\newtheorem{corollary}[theorem]{Corollary}
	\newtheorem{lemma}[theorem]{Lemma}
	\newtheorem{proposition}[theorem]{Proposition}
	\theoremstyle{definition}
	\newtheorem{definition}[theorem]{Definition}
	\theoremstyle{remark}
	\newtheorem{remark}{\bf Remark}
	\newtheorem*{remark*}{\bf Remark}
	\newtheorem{example}[theorem]{Example}

	\newcommand{\THH}{\mathrm{THH}}
	\newcommand{\TAQ}{\mathrm{TAQ}}
	\newcommand{\AQ}{\mathrm{AQ}}
	\newcommand{\sslash}{\mathbin{/\mkern-6mu/}}
	
\title{On Topological Andr\'e-Quillen homology of Eilenberg-MacLane spectra}
\author{Cyril Barlasov}

\email{kirill.barlasov02@gmail.com}
\begin{abstract}
	Based on the work of Dundas, Lindenstrauss and Richter we compute the topological Andr\'e-Quillen homology with reduced coefficients for Eilenberg-MacLane spectra such as $H\mathbb{Z}$ and $H\mathbb{Z}/p^n$. The case of $H\mathbb{F}_p$ was settled in an unpublished work of Basterra and Mandell, which was refined later by Brantner and Mathew in the context of spectral partition Lie algebras. Our approach is similar to Cartan's calculation of the Steenrod algebra, and eventually shorter. We also present some computations relative to the $H\mathbb{Z}$ base.
\end{abstract}

	\maketitle
\section{Introduction}

In the late 90's Basterra \cite{B1} has constructed a structured version of Postnikov towers for $\mathbb{E}_{\infty}$-ring spectra, relying on a new cohomology theory $\TAQ^*$. It is a literal generalization in a "brave new" sense of Andr\'e-Quillen cohomology for ordinary commutative rings, namely the derived cotangent complex. Topological version, however, admits another description as a stabilization of higher order topological Hochschild homology $\THH^{[n]}$. These are also considered as a proper tools for studying iterated $K$-theory via trace methods.

We review the necessary and relevant preliminaries on $\THH$ in \textsection\ref{sTHH} and $\TAQ$ in \textsection\ref{sTAQ}. A short recollection of ordinary Andr\'e-Quillen homology, serving as motivation for the latter, can be found in \textsection\ref{sGamma}.

The goal of the present work is to clarify Lazarev's statement on $\TAQ^*(H\mathbb{F}_p,\mathbb{S})$ in \cite{Laz}. We were unable to complete the sketch of argument on our own, though he gives a correct answer for ${p=2}$.
The corresponding result is attributed to Basterra and Mandell (unpublished in 1999), and was established via Miller type spectral sequence. From a modern perspective, it is part of the computation of homotopy groups of free spectral partition Lie algebras, due to Brantner and Mathew, see \textsection\ref{sLie} and \cite{BrMat}.
\begin{equation*}\tag{\ref{stabFp}}
	\TAQ_*(H\mathbb{F}_p,\mathbb{S})\simeq \pi_{-2-*}\big(\text{Lie}^{\pi}_{\mathbb{F}_p}(\Sigma^{-3}\mathbb{F}_p)\big)
	\simeq\mathbb{F}_p\bigg\langle(a_1,...,a_k)\bigg|\substack{a_i=0\text{ or }1\ \text{(mod }2p-2); 
		\\
		a_1\geq { 4}(p-1);\ a_{i}\geq pa_{i-1}\text{ for }2\leq i\leq k;\\
		\mathrm{deg}=1+a_1+...+a_k.} \bigg\rangle
\end{equation*}

We give a new argument in \textsection\ref{sHigherThh}, based on the work of Dundas, Lindenstrauss and Richter \cite{DLR} on higher topological Hochschild homology $\THH^{[n]}$. Such a solution turns out to be completely parallel to a classical problem on the homology of Eilenberg-MacLane spaces $K(\mathbb{Z}/p,n)$. Moreover, the same method yields
\begin{equation*}\tag{\ref{TAQint}}
	\TAQ_*(H\mathbb{Z},\mathbb{S};H\mathbb{F}_p)
	\simeq\mathbb{F}_p\bigg\langle (a_0, a_1,...,a_k) \bigg|\substack{a_i=0\text{ or }1\ \text{(mod }2p-2) \text{ for } 1\leq i\leq k;\\
		a_1\geq { 2(p+a_0)}(p-1);\ a_{i}\geq pa_{i-1}\text{ for }2\leq i\leq k;\\
		a_0=0\text{ or }1;\quad \mathrm{deg}=2p+a_0+a_1+...+a_k.}\bigg\rangle
\end{equation*}

The full answer with absolute integral coefficients amounts to the knowledge of higher Bockstein operators, which we didn't manage to determine due to the lack of any transparent useful structure on $\TAQ$, except for the module one. In contrast, working out Bocksteins in Lie algebras framework might be feasible, but also more difficult. Apparently, it has not been carried out yet, even the calculation with reduced coefficients alone. Without any strong evidence, our expectation is that the situation might be similar to that of integral homology of symmetric groups, i.e.\ the answer is algorithmic in nature. Nevertheless, it is not difficult to show that the groups $\TAQ_*(H\mathbb{Z}, \mathbb{S})$ are rationally trivial, therefore do not contain a free $\mathbb{Z}$-summands (see Proposition \ref{Zrational}), which allows to extract the ranks of $p$\nobreakdash-power torsion from our expression. One obtains same expressions for $\TAQ$ of rings of integers with appropriate coefficients in residue fields, though we do not formulate these. Our other result is
\begin{equation*}\tag{\ref{taqzpn}}
	\TAQ_*(H\mathbb{Z}/p^n,\mathbb{S};H\mathbb{F}_p)
	\simeq \TAQ_*(H\mathbb{Z},\mathbb{S};H\mathbb{F}_p)\oplus \mathbb{F}_p\bigg\langle (b_1,...,b_k) \bigg|\substack{b_i=0\text{ or }1\ \text{(mod }2p-2) \text{ for } 1\leq i\leq k;\\
		b_1\geq  2(p-1);\ b_{i}\geq pb_{i-1}\text{ for }2\leq i\leq k;\\
		\mathrm{deg}=1+b_1+...+b_k.}\bigg\rangle
\end{equation*}
deduced from the work \cite{BHLKRZ} on the generalized Brun juggling in Loday construction. The same issue with determining Bocksteins stands.

The strategy is as follows: given the algebra $\THH_*(R; H\mathbb{F}_p)$, one consecutively runs the Eilenberg-Moore spectral sequence to get to the higher order $\THH^{[n]}_*$ for all $n$. Then it is a matter of combinatorics to obtain $\TAQ_*$ as a colimit, since stabilization maps are induced on the second pages, and kill decomposables (this observation is essentially due to Moore, see Proposition \ref{kill}). Generally, that puts an upper bound for $\TAQ_*$ assuming degenerations of these spectral sequences. In named cases they actually degenerate, and an involved combinatorial part was already covered by Cartan in \cite{Car}.

Some satisfying negative result is in order. B\"okstedt's original approach to $\THH_*(H\mathbb{F}_p)$ relied on an elaborate spectral sequence considerations, but there is a shortcut\footnote[1]{If one wishes to view it that way} using Mahowald type theorems that exhibit Eilenberg-MacLane's as Thom spectra, together with the result of Blumberg, Cohen and Schlichtkrull on the $\THH$ of Thom spectra, see \cite{Kit}. However, this approach does not apply to our situation. While there exists a similar theorem due to Basterra and Mandell for the $\TAQ$ of $\mathbb{E}_{\infty}$\nobreakdash-Thom spectra \cite{BM}, Mahowald's theorem establishes only an $\mathbb{E}_{2}$-equivalence, so the former result is of no use. We do not pursue a precise nonrealizability statement here, yet it seems within reach.

Algebraically indoctrinated reader may also like to consult the work of Kaledin \cite{Kld}. In particular, it presents a topology-distilled proof of the B\"okstedt periodicity and a calculation of the Steenrod algebra. We believe that there might exist an algebraic solution to $\TAQ_*(H\mathbb{F}_p,\mathbb{S})$ along the lines of loc.\ cit.,\ since our main tool here -- the stabilization procedure -- can very well be implemented algebraically, and the other significant input is the B\"okstedt periodicity itself.

Perhaps less interesting, in \textsection\ref{sGamma} we investigate the case over the $H\mathbb{Z}$ base instead of the sphere $\mathbb{S}$. Note, that it naturally reduces to algebra, in fact, to functor homology. By means of the Richter spectral sequence \cite{Rich} we compute

\begingroup
\renewcommand\thetheorem{\ref{TAQ(Z/p^n,Z)}}
\begin{proposition}There is an isomorphism of graded $\mathbb{F}_p$-modules
	$$\TAQ_*(H\mathbb{Z}/p^n,H\mathbb{Z})\simeq\Sigma^2\mathcal{A}_p/\beta$$
	where $\mathcal{A}_p$ is the mod $p$ Steenrod algebra. The quotient is taken wrt the two-sided ideal generated by the Bockstein operation $\beta$. \underline{There is a copy of $\mathbb{F}_p$ in degree $1$ and not $2$}.
\end{proposition}
\endgroup

Just as for smooth algebras, spectral sequence degenerates, and everything boils down to the knowledge of the integral Steenrod algebra, which is also due to Cartan. We didn't find a suitable quantitative form of it in the literature, so include a brief proof here, see Proposition \ref{intSt}. 

As a particular point-set model for the category of spectra we use $\mathbb{S}$-modules and algebras of \cite{EKMM} throughout the paper.  

\vspace{3mm}

\noindent\textbf{Acknowledgements.} I am grateful to Fedor Vylegzhanin for teaching me relevant algebraic topology, and to Dmitry Kaledin for his patience, advise and critique.

\section{Topological Hochschild homology}
\label{sTHH}

For the sake of consistency with references, we will use the definition of $\THH$ as a Loday construction, like in \cite{Veen}. This is a particular model for coherent homotopy colimits of the certain form.
\begin{definition}
	Given a simplicial set $X$ and a commutative ring spectrum $R$, the \textit{Loday construction} is a categorical tensor in the category of $\mathbb{S}$-modules $\Lambda_X R:= X\otimes R$.
\end{definition}

More explicitly, consider the functor, defined on finite sets as $Y\to\wedge_{y\in Y}R$, which multiplies components of preimages for surjective maps, and inserts identities for injective. Prolong its domain to simplicial sets. Homotopy colimit of the resulting diagram is the Loday construction. For example, the two-sided bar construction on $R$ coincides with $\Lambda_{\Delta^1}R$, where $\Delta^1$ is the standard 1-simplex with $q+2$ elements in dimension $q$. 

\begin{theorem}[\cite{Veen}]\label{lodprop} Loday construction enjoys the following properties
	
	(i) Equivalence $X\xrightarrow{\sim}Y$ of simplicial sets induces an equivalence $\Lambda_X R\xrightarrow{\sim}\Lambda_YR$.
	
	(ii) Given any simplicial set $Y$, there is an equivalence $\Lambda_Y\Lambda_X R\simeq \Lambda_{X\times Y} R$.
	
	(iii) Given a cofibration $L\to X$ and a map $L\to K$ of simplicial sets, there is an equivalence $\Lambda_{X\coprod_L K} R\simeq \Lambda_X R \wedge_{\Lambda_L R} \Lambda_K R$.\qed
\end{theorem}

\begin{definition}
	 $\THH(R):=\Lambda_{S^1}R$, where $S^1$ is the simplicial circle.
\end{definition}

We note that $\THH(R)$ is a commutative $R$-Hopf algebra. The inclusion of the basepoint induces a unit map $\eta\colon R\simeq \Lambda_{\mathrm{pt}} R\to \Lambda_{S^1} R\simeq \THH(R)$. Contraction $S^1\to \mathrm{pt}$ induces an augmentation $\varepsilon\colon \THH(R)\simeq\Lambda_{S^1}R\to \Lambda_{\mathrm{pt}}R\simeq R$. Pinching of two points together $\psi : S^1\to S^1\vee S^1$, folding $\mu : S^1\vee S^1\to S^1$ and reflection $\chi : S^1\to S^1$ induce the coproduct, product, and antipode resp. Passing to homotopy groups, $\THH_*(R)$ becomes a $\pi_*(R)$-Hopf algebra.

Replacing the smash product with wedge in Loday construction results in a spectrum $\vee_XR=X_+\wedge R$, together with a natural map $\omega_X:\vee_XR\to\Lambda_XR$. Composition of  $$\omega_X:X_+\wedge \Lambda_YR\to\Lambda_X\Lambda_YR\simeq\Lambda_{X\times Y}R$$
with the map induced by the quotient $X\times Y\to X\wedge Y$ produces the natural map
$$ \widehat{\omega}_X:X_+\wedge \Lambda_Y R\to \Lambda_{X\wedge Y}R.$$

\begin{definition}[Suspension map $\sigma_*$]\label{susp} The composition of $\widehat\omega_{S^1}$ with a chosen splitting of $S^1_+\simeq S^0\vee S^1$ induces a map 
	$$\widehat\omega_{S^1}:\pi_*(S^1_{+}\wedge \Lambda_XR)\simeq H_*(S^1)\otimes\pi_*(\Lambda_XR) \to \pi_*(\Lambda_{S^1\wedge X}R)$$
Define
$$\sigma_*: \pi_*(\Lambda_XR)\to \pi_*(\Lambda_{S^1\wedge X}R)$$
by taking $x\in \pi_*(\Lambda_XR)$ to the image of $[S^1]\otimes x$ under $\widehat\omega_{S^1}$ for a chosen generator of $\widetilde{H}_*(S^1)$.
\end{definition}

\begin{proposition}\label{kill} The suspension map $$\sigma_*:\pi_*(\Lambda_{S^n}R)\to\pi_*(\Lambda_{S^1\wedge S^n}R)$$
has all $\pi_*R$-algebra decomposables of $\pi_*(\Lambda_{S^n}R)$ in its kernel.
\end{proposition}
\begin{proof}
	Note that both the source and the target are commutative $\pi_*R$-Hopf algebras.
	
	By the Proposition 8.5 of \cite{Veen}, the suspension map $\sigma_*$ sends an element $z\in\pi_{m-1}(\Lambda_{S^n} R)$ to the class $[z]$ in the first page of the skeletal filtration spectral sequence. That is, the reduced bar complex ${E_{1,m}^1\simeq B_1(\pi_*R,\pi_*(\Lambda_{S^n}R),\pi_*R)_m\simeq\pi_m(\Lambda_{S^n}R)}$. But the differential is ${d[x|y]=\pm [xy]}$, so decomposable elements of the form $[xy]$ are boundaries and vanish on the second page $E^2$ of the spectral sequence.
\end{proof}

\section{Topological Andr\'{e}-Quillen homology}
\label{sTAQ}

For $A$ and $B$ a commutative $\mathbb{S}$-algebras, we denote by $\textrm{CAlg}_B/A$ the category of commutative $A$-augmented $B$-algebras and by $\text{CAlg}^{\text{nu}}_A$ the category of commutative nonunital $A$-algebras.

\begin{definition}
	\textit{Augmentation ideal} $I_A(X)\in\text{CAlg}^{\text{nu}}_A$ of $X\in \textrm{CAlg}_A/A$ is the fiber of the augmentation map 
	$$I_A(X)\to X\xrightarrow{\varepsilon}A$$
	\textit{Module of indecomposables} $Q_A(N)\in \mathrm{Mod}_A$ of $N\in \text{CAlg}^{\text{nu}}_A$ is the cofiber of the multiplication map 
	$$N\wedge_A N\xrightarrow{\mu}N\to Q_A(N)$$
\end{definition}

As shown in \cite{B1}, the split square-zero extension functor $B\vee -\colon \text{Mod}_B\to \text{CAlg}_A/B$ is right adjoint to $QI(-\wedge_A B)\colon \text{CAlg}_A/B\to \text{Mod}_B$ in the derived sense.

\begin{definition}\textit{Cotangent complex} of $B\in\text{CAlg}_A$ is the $B$-module $$\Omega_{B|A}:=QI(B\wedge_A B)$$
	\textit{Topological Andr\'e-Quillen (co)homology} of $B$ over $A$ with coefficients in module $M\in\text{Mod}_B$ is  $$\TAQ_*(B,A;M):=\pi_*(\Omega_{B|A}\wedge_B M)$$ $$ \TAQ^*(B,A;M):=\pi_{-\ast}\text{Map}_B(\Omega_{B|A},M).$$
\end{definition}

There is a different definition of $\TAQ$ as a stabilization. The adjunction $(\Sigma, \Omega)$ in the category of pointed spaces leads in the usual way to the category of sequential spectra. The category $\text{CAlg}^{\text{nu}}_B$ has $E:=-\otimes S^1$ and $\Omega:=(-)^{S^1}$ as a suspension and loops functors respectively, i.e.\ a pushout and pullback along initial object. While $E(C)=B\wedge_C B$ differs from $\Sigma$ upon forgetting algebra structure to $\text{Mod}_B$, the loops $\Omega(C)\simeq \Sigma^{-1}C\simeq \text{Map}(\Sigma^\infty S^1, C)$ coincide as modules. In the latter case, the $B$-algebra structure is induced from the second argument of the $\mathrm{Map}$ functor (see VII.2.8, 3.10 \cite{EKMM}). It is now standard to define the category of $(E,\Omega)$-spectra in $\text{CAlg}^{\text{nu}}_B$, similarly for the equivalent augmented case $\text{CAlg}_B/B$.

\begin{theorem}[see 3.5 in \cite{BM}]\label{stab} All three pairs of adjoint functors below are equivalences of categories. 
	\begin{align*}
		\xymatrix{
			\mathrm{Mod}_B \ar@<1ex>[r]^{\Sigma^\infty} & \mathrm{Sp}(\mathrm{Mod}_B) \ar[l]^{(-)_0} \ar@<1ex>[r]^{\mathbb{P}^{\mathrm{nu}}} & \mathrm{Sp}(\mathrm{CAlg}^{\mathrm{nu}}_B) \ar@<1ex>[r]^{B\vee -} \ar[l] & \ar[l]^I \mathrm{Sp}(\mathrm{CAlg}_B/B) 
		}
	\end{align*}
	where the rightmost is induced from the equivalence of underlying categories. Leftmost is due to the fact that the modules category is already stable. The middle one is the stabilization of the free-forgetful adjunction, such that $\mathbb{P}^{\mathrm{nu}}(M)=\vee_{k>0}M^{\wedge_B k}\sslash\mathfrak{S}_k$ is the free commutative nonunital algebra on a $B$\nobreakdash-module and $\sslash$ denotes homotopy quotient.\qed
\end{theorem}

Under these equivalences, the composite functor takes a commutative $\mathbb{S}$-algebra $B$ to an augmented algebra $B\wedge B\in \mathrm{CAlg}_B/B$ with $\mu:B\wedge B\to B$ as augmentation, then to the $E$-suspension spectrum $\big\{ m\mapsto E^m(B\wedge B)\big\}$, further to the augmentation ideal $\big\{ m\mapsto IE^m(B\wedge B)\big\}$, and finally to the $B$-module  $$\TAQ(B,\mathbb{S})\simeq\underset{n}{\text{hocolim}}\  \Sigma^{-n}IE^n(B\wedge B).$$

The last equivalence is due to the fact that the indecomposables functor $Q$ coincides with the stabilization. The idea is as follows: since everything in sight are adjoint functors, they commute with colimits, and resolving algebras with the bar construction, the claim reduces to free algebras. In this case we have
\begin{theorem}[\cite{BM}]\label{taqfree} Given $X\in \mathrm{Mod}_R$, and $\mathbb{P}_R$ a free augmented $R$-algebra functor, there is an equivalence of $\mathbb{P}_RX$-modules
	\[
	\pushQED{\qed} 
	\TAQ(\mathbb{P}_RX, R)\simeq \mathbb{P}_R X\wedge_R X.\qedhere
	\popQED
	\]
\end{theorem}

As a simple and useful corollary of the structured Postnikov towers theory \cite{B1} we recall Mandell's Lemma. It states that among \textit{commutative} augmented algebras, there is a unique up to weak equivalence with exterior algebra as its homotopy groups. 

\begin{theorem}[see 2.1 in \cite{DLR}]\label{ManLemma} Let $k$ be a commutative ring and $C\in \mathrm{CAlg}_{Hk}/Hk$ a commutative augmented $Hk$-algebra. Suppose that there is an isomorphism of $k$-algebras $\pi_\ast C\simeq E_k(x)$, with $\mathrm{deg}(x)=m>0$. Then $C\simeq Hk\vee \Sigma^m Hk$.\qed
\end{theorem}

\section{Iterated bar construction}
\label{sHigherThh}

\begin{definition} \textit{Higher topological Hochschild homology} of $R\in \mathrm{CAlg}_{\mathbb{S}}$ is
	$$\THH^{[n]}(R):=\Lambda_{S^n}R$$
	the Loday construction for $n$-dimensional sphere.
\end{definition}

Analogously to the case of circle $S^1$ it is an $R$-Hopf algebra. As an $R$ we still consider $\mathbb{E}_{\infty}$-ring spectrum, hence $\THH^{[n]}(R)$ is also an $\mathbb{E}_{\infty}$-ring spectrum. Forming the sphere as a homotopy pushout $S^n\simeq D^{n}\coprod_{S^{n-1}} D^{n}\simeq \mathrm{pt}\coprod_{S^{n-1}}\mathrm{pt}$, by the property \ref{lodprop} (iii), we have equivalence
$$ \THH^{[n]}(R)\simeq R\wedge_{\THH^{[n-1]}(R)}R.$$

We will be interested in the case of $R\simeq H\mathbb{F}_p$ at the moment. The Eilenberg-Moore spectral sequence (see \cite{EKMM}) for the homotopy groups then looks like
$$
E^2_{*,*}=\text{Tor}^{\THH^{[n-1]}_\ast(H\mathbb{F}_p)}_{*,*}(\mathbb{F}_p,\mathbb{F}_p)\Rightarrow \THH^{[n]}_{*}(H\mathbb{F}_p).
$$
Theorem of Dundas, Lindenstrauss and Richter states that each time it degenerates at the second page and reduces to an algebraic problem of calculating the iterated $\text{Tor}$ functor, starting from the basic case $\THH_*(H\mathbb{F}_p)\simeq\mathbb{F}_p[\mu]$. 

\begin{theorem}[\cite{DLR}]\label{DLR} Let $B^1=P(\mu)$ be the polynomial $\mathbb{F}_p$-algebra on a generator $\mu$ of degree 2. Inductively, let $B^n=\mathrm{Tor}^{B^{n-1}}_*(\mathbb{F}_p, \mathbb{F}_p)$. Then
	\[
	\pushQED{\qed} 
	\THH^{[n]}_*(H\mathbb{F}_p)\simeq B^n.\qedhere
	\popQED
	\]
\end{theorem}

Recall the following about the $\mathrm{Tor}$ functor. For a graded commutative algebra $A$ over a commutative ring $k$, the two-sided bar construction $B_\ast(k,A,k)$ is a CDG-algebra over $k$ with the shuffle product. $\text{Tor}^A_*(k,k)$ is the homology ring of $B_\ast(A)$. We will consider the case $k=\mathbb{F}_p$, and as $A$ we will take algebras $P(x),\ E(x) \text{, and } \Gamma(x)$ i.e.\ a polynomial, exterior and divided powers respectively. Further we assume $p$ being odd, then the algebra $\Gamma(x)$ reduces to truncated polynomials $P_p(x):=k[x]/x^p$ due to the K\"unneth theorem and choice of base ring, that is $\Gamma(x)=\otimes_{n=0}^{\infty}P_p\big(\gamma_{p^n}(x)\big)$.

\begin{lemma}[see \cite{BHLKRZ}]\label{TorAlg}\
	
	(i) For a polynomial algebra $P(x)$ with $\mathrm{deg}(x)$ even
	$$\mathrm{Tor}^{P(x)}_*(\mathbb{F}_p,\mathbb{F}_p)\simeq E(\sigma x)$$
	where $\mathrm{deg}(\sigma x)=\mathrm{deg}(x)+1$.
	
	(ii) For an exterior algebra $E(x)$ with $\mathrm{deg}(x)$ odd
	$$\mathrm{Tor}^{E(x)}_*(\mathbb{F}_p, \mathbb{F}_p)\simeq \Gamma(\rho^0x)$$
	where $\mathrm{deg}(\rho^0x)=\mathrm{deg}(x)+1$.
	
	(iii) For a truncated polynomials $P_p(x)$ with $\mathrm{deg}(x)$ even
	$$\mathrm{Tor}^{P_p(x)}_*(\mathbb{F}_p, \mathbb{F}_p)\simeq E(\sigma x)\otimes \Gamma(\phi^0x)$$
	where $\mathrm{deg}(\sigma x)=\mathrm{deg}(x)+1$ and $\mathrm{deg}(\phi^0 x)=2+p\cdot \mathrm{deg}(x)$.\qed
\end{lemma}

Morover, it is possible to construct a multiplicative quasi-isomorphisms between bar construction algebras and their homologies in all these cases. Hence, iteration of the bar construction, starting with the polynomial algebra, gives successively exterior, then divided, and the product of exterior and divided algebras afterwards. For convenience, we switch the notation for the generators in the decomposition of divided powers to $\rho^ex:=\gamma_{p^e}(\rho^0x)$, and ${\phi^ex:=\gamma_{p^e}(\phi^0x)}$.
\begin{figure}[h]\label{flow}
	\caption{The algebras $B^n$ for $1\leq n\leq 4$}
	\begin{align*}
		\xymatrix{
			&&&\underset{k\geq 0}{\bigotimes}E(\sigma\rho^k\sigma\mu)\ar[r]&\dots\\
			P(\mu)\ar[r] & E(\sigma\mu) \ar[r] & \Gamma(\rho^0\sigma\mu)\simeq\underset{k\geq 0}{\bigotimes}P_p(\rho^k\sigma\mu) \ar[ru]\ar[rd] & &\dots\\
			&&&\underset{k\geq 0}{\bigotimes}\Gamma(\phi^0\rho^k\sigma\mu)\simeq \underset{i,k\geq 0}{\bigotimes}P_p(\phi^i\rho^k\sigma\mu) \ar[ru] \ar[rd] &\\
			&&&&\dots\\
		}
	\end{align*}
\end{figure}
\begin{definition}\label{adm}In the alphabeth $\{\mu,\sigma,\rho^k,\phi^k\}$ where $k\geq 0$, \textit{admissible word} is the following:
	\begin{itemize}
		\item It ends with $\mu$
		\item If $\mu$ is preceded by a letter, it is $\sigma$
		\item If $\sigma$ is preceded by a letter, it is $\rho^k$
		\item If $\rho^k$ or $\phi^k$ is preceded by a letter, it is $\sigma$ or $\phi^l$.
	\end{itemize}
	The admissible word is \textit{stable} if it starts with one of the letters $\mu,\sigma, \rho^0, \phi^0$. Degrees are defined recursively: 
\begin{equation*}
	\begin{aligned}[c]
		\text{deg}(\mu)&=2\\
		\text{deg}(\sigma x)&=1+\text{deg}(x)
	\end{aligned}
	\qquad\ \qquad
	\begin{aligned}[c]
		\text{deg}(\rho^k x)&=p^k\big(1+\text{deg}(x)\big)\\
		\text{deg}(\phi^k x)&=p^k\big(2+p\cdot \text{deg}(x)\big).\\
	\end{aligned}
\end{equation*}
\end{definition}

\begin{proposition}\label{B^n}Algebra $B^n$ is the tensor product of exterior algebras on generators indexed by stable admissible words of length $n$ of odd degree, and of divided power algebras on those of even degree. 
	\[
	\pushQED{\qed} 
	B^n\simeq \left(\underset{|w| \equiv 1}{\bigotimes}E(w)\right)\otimes\left(\underset{|w| \equiv 0}{\bigotimes}\Gamma(w)\right)\qedhere
	\popQED
	\]
\end{proposition}

This summarizes the relevant results of \cite{Veen} and \cite{BHLKRZ}. However, we would like to get something more than just a simple algorithm. 

Let us first recall Cartan's approach to the classical problem on (co)homology of Eilenberg-MacLane spaces.  Cohomology algebra of a lens space \textit{with} $\mathbb{F}_p$ \textit{coefficients} is $H^*K(\mathbb{Z}/p,1)\simeq P(\phi^0)\otimes E(\sigma)$
the product of the polynomial and exterior algebras with generators of degree $2$ and $1$, respectively. Dually, the homology algebra is
$$H_*K(\mathbb{Z}/p,1)\simeq \Gamma(\phi^0)\otimes E(\sigma).$$
Further, in following cases there are quasi-isomorphisms between chains DG-algebras and their homologies as graded algebras with trivial differential, that is
$$C_*K(\mathbb{Z}/p,n)\xrightarrow{\sim}H_*K(\mathbb{Z}/p,n)$$
so we could interchange them under the bar  construction on the right hand side of
$$H_*K(\mathbb{Z}/p,n+1)\simeq H_*B_\bullet\big(C_*K(\mathbb{Z}/p,n)\big).$$
Isomorphism in this form is due to Milgram, but the relevant case was exploited much earlier by Eilenberg and MacLane themselves. Schematically, $$H(n+1) = \textrm{Tor}_*^{H(n)}(\mathbb{F}_p, \mathbb{F}_p), \quad\textrm{where}\ H(n) := H_*K(\mathbb{Z}/p, n).$$ We arrive at calculating the iterated bar construction of the exterior and divided power algebras, which was adressed by Lemma \ref{TorAlg}.

\begin{definition}\label{adm_p}
	The sequence of nonnegative integers $I=(a_0,...,a_k)$ is \textit{admissible}, if
	\begin{center}
		$a_0=0\text{ or }1$;

		$a_i=0\text{ or }1$ (mod $2p-2$) for $0\leq i\leq k;$

		$a_1\geq 2p-2$;\quad $a_{i+1}\geq pa_i$ for $1\leq i\leq k-1$.
	\end{center}
	degree is defined as $|I|=a_0+a_1+...+a_k$.
\end{definition}

\begin{theorem}[see \textsection6 in \cite{Car}]\label{Cartan} Let $n\geq 1$ and $p$ odd, then 
	$$H_*K(\mathbb{Z}/p,n)\simeq \left(\underset{\mathrm{odd}}{\bigotimes}E(I)\right)\otimes\left(\underset{\mathrm{even}}{\bigotimes}\Gamma(I)\right)$$
	The $\text{mod } p$ homology algebra of $K(\mathbb{Z}/p,n)$ is a tensor product of exterior algebras on odd degree generators, and divided powers on even. Generators are indexed by admissible sequences $I=(a_0,...,a_k)$ of degrees $|I|+n$, satisfying $pa_k<(p-1)(n+|I|)$.
\end{theorem}
\begin{proof}
	Bijection with words arising from the iterated bar construction is as follows. Let $\alpha$ be an admissible word in the sense of Definition \ref{adm}, with $\mu$ dropped. Cartan defines them differently, accounting for symbol $\gamma_{p^n}$ as a separate word $\gamma_{p}...\gamma_{p}$, we follow this convention for a moment.

	Consider the sequence (possibly empty) of the letters of $\alpha$, distinct from $\sigma$ or $\rho^0$, enumerate them from right to left. For $\alpha_i$, the $i$-th of these, let $2k_i$ be the degree of \textit{subword} $\beta_i$ \textit{inside the whole} $\alpha$, located strictly to the right of $\alpha_i$, it is evidently an even integer by the Definition \ref{adm}. Then the word $\alpha_i\beta_i$ has degree $2pk_i$ if $\alpha_i=\gamma_p$, and degree $2pk_i+2$ if $ \alpha_i=\phi^0$. Let $a_i$ be the "modified" difference in degrees of $\alpha_i\beta_i$ and $\beta_i$, so that
	\begin{align*}
		a_i=2k_i(p-1)+\varepsilon_i\quad\text{ with }\varepsilon_i=0&\text{ if }\alpha_i=\gamma_p;\\
		\quad\qquad \varepsilon_i=1&\text{ if }\alpha_i=\phi^0.
	\end{align*}
	This defines $a_i$ for $i\geq1$. Let $a_i=0$ if $i$ is greater than the number of occurrences of $\gamma_p$ and $\phi^0$ in $\alpha$. 
	
	The word, which contains only $\sigma$ and $\rho^0$, defines a sequence with $a_i=0$. It also follows directly from the definition, that $a_i=0\text{ or }1$ (mod $2p-2$). As for the condition $a_{i+1}\geq pa_i$, note that $\mathrm{deg}(\alpha_i\beta_i)=2k_ip+2\varepsilon_i$, then $2k_{i+1}\geq 2k_{i}p+2\varepsilon_i$ for any $i\leq k-1$.  
	
	The excess inequality $pa_k<(p-1)(n+|I|)$ restricts the length of an admissible word by $n$. The factor $n+|I|$ equals $\mathrm{deg}(\alpha)$, and by the definition of $a_i$, its necessity is to be checked in two cases: when the leftmost letter is $\phi^0$ or $\gamma_p$. 
	
	Now, the homology algebra of $K(\mathbb{Z}/p,n)$ has two kinds of generators, which come from the exterior \textit{or} divided algebra. In the first case generators end on $...\rho^0\sigma$, and on $...\phi^0$ in the second, so the minimum for $a_1$ is attained when the next letter is $\gamma_p$, hence the condition $a_1\geq 2(p-1)$. The generators of the second kind are accounted by an additional number $a_0$ with the value $1$.	
	
	It is straightforward to check that such $(a_1,...,a_k)$ satisfies the conditions of admissibility \ref{adm_p} and defines a word uniquely. More precisely, Cartan gives an algorithm for constructing a word from a sequence.
\end{proof}

Now, as a graded module, the dual mod p Steenrod algebra is the colimit 
\begin{equation*}
	\mathcal{A}_p\simeq \underset{n}{\mathrm{colim}}\ \Sigma^{-n}\widetilde{H}_*\big(K(\mathbb{Z}/p,n);\mathbb{F}_p\big)
\end{equation*}
taken over $\Sigma K(\mathbb{Z}/p,n)\xrightarrow{\sigma}K(\mathbb{Z}/p,n+1)$ the counit maps of the $(\Sigma,\Omega)$-adjunction (compare to \ref{susp}). 

Let's return to $\THH^{[n]}(H\mathbb{F}_p)$. The proof of Theorem \ref{Cartan} applies verbatim, with the only difference being ${B^1\simeq P(\mu)}$. 
\begin{proposition}\label{THHadm} For $n>1$ and odd $p$
	$$\THH^{[n]}_*(H\mathbb{F}_p)\simeq \left(\underset{\mathrm{odd}}{\bigotimes}E(I)\right)\otimes\left(\underset{\mathrm{even}}{\bigotimes}\Gamma(I)\right)$$
	Generators are indexed by admissible sequences $I=(a_1,...,a_k)$ such that
	\begin{center}
		$a_i=0$ $\mathrm{or}$ $1$ $(\mathrm{mod}$ $2p-2)$ $\mathrm{for}$ $1\leq i\leq k;$
		
		$a_1\geq {\bf 4}(p-1);$\quad $a_{i+1}\geq pa_i$ $\mathrm{for}$ $1\leq i\leq k-1;$
		
		$pa_k<(p-1)(n+|I|);$
	\end{center}
	and have degrees $n+|I|$, where $|I|=1+a_1+...+a_k$.
\end{proposition}
\begin{proof}
	Using Theorem \ref{DLR}, we've got to start induction from $n=2$ because the polynomial algebra for $n=1$ does not fit here uniformly. Hence a shift in the degree of the sequence.
	
	Bijection goes as in the Theorem \ref{Cartan}, but admissible words will end up with ...$\rho^0\sigma\mu$, which plays the role of a homology class $...\rho^0\sigma\in H_*K(\mathbb{Z}/p,n)$ for $n\geq 2$, and has degree $4$. Hence a bold four in the restriction on $a_1$.
	
	Indices $a_i$ are defined by the same procedure, so they satisfy the rest of the remaining conditions. The number $|I|+n$ equals to the degree of an admissible word, which corresponds to a sequence $(a_1,...,a_k)$.
\end{proof}

Similarly to the passage from Theorem \ref{Cartan} to Steenrod algebra, we will use the above proposition and the definition of $\TAQ$ as a stabilization
$$\TAQ(R,\mathbb{S})\simeq\underset{n}{\mathrm{hocolim}}\ \Sigma^{-n}IE^n(R\wedge R)=:\underset{n}{\mathrm{hocolim}}\ \Sigma^{-n}\overline{\THH}^{[n]}(R).$$

The dash over $\overline{\THH}^{[n]}(R)$ means the augmentation ideal of the $\mathbb{S}$-algebra. Its homotopy groups are literally the augmentation ideal when $R\simeq Hk$, due to the fibration long exact sequence. Therefore 
\begin{equation}\tag{$\ast$}\label{ast}
	 \TAQ_{\ast}(R)\simeq \underset{n}{\mathrm{colim}} \left(\overline{\THH}_{\ast+1}(R)\xrightarrow{\sigma_*}\overline{\THH}_{\ast+2}^{[2]}(R)\xrightarrow{}...\xrightarrow{}\overline{\THH}_{\ast+n}^{[n]}(R)\xrightarrow{}...\right)
\end{equation}

Colimit is taken over the maps induced by the following composition (see \ref{susp}) 
\begin{align*}
	\xymatrix@R=0.5em{
		\sigma_*:S^1\wedge\Lambda_{S^n}R\ar[r]^(0.57){\widehat{\omega}_{S^1}}& \Lambda_{S^1\times S^n}R \ar[r]^{\Lambda_\pi} & \Lambda_{S^1\wedge S^n}R\\
		 \Sigma \THH^{[n]}(R) \ar@{}[u]|*=0[@]_{\simeq} & & \ar@{}[u]|*=0[@]{\simeq} \THH^{[n+1]}(R)
	}
\end{align*}
of nonreduced suspension $\widehat{\omega}_{S^1}$ in $\THH$ and the maps $\Lambda_\pi$, induced by the projection of product onto smash $\pi: S^1\times S^n\to (S^1\times S^n)/(S^1\vee S^n)\simeq S^1\wedge S^n$.
Recall that $\sigma_*$ kills decomposables by Proposition \ref{kill}, and by the description of this map given in its proof, effect on the generators, i.e.\ admissible words $x\in \THH_*^{[n]}(H\mathbb{F}_p)$ is
$$
\sigma_*(x)=
\begin{cases}
	\sigma x& \mathrm{deg}(x) \text{ even},\\
	\rho^0 x& \mathrm{deg}(x) \text{ odd}.
\end{cases}
$$
The element of colimit is an equivalence class, with the identification of words differing only in the string of $\{...\sigma\rho^0\sigma\rho^0...\}$ on the left. By the construction of admissible sequences, adding $\sigma$ or $\rho^0$ on the left doesn't change the numbers $a_i$, so $\sigma_*(I)=I$.

\begin{proposition}[compare to \textsection\ref{sLie}]\label{stabFp} The graded $\mathbb{F}_p$-module $$\TAQ_*(H\mathbb{F}_p,\mathbb{S})$$ has a basis indexed by sequences $I=(a_1,...,a_k)$ such that
	\begin{center}
		$a_i=0$ $\mathrm{or}$ $1$ $(\mathrm{mod}$ $2p-2)$ $\mathrm{for}$ $1\leq i\leq k;$
		
		$a_1\geq 4(p-1);$\quad $a_{i+1}\geq pa_i$ $\mathrm{for}$ $1\leq i\leq k-1;$
	\end{center}
	of degrees $|I|=1+a_1+...+a_k$.
\end{proposition}
\begin{proof} 
	Though $\sigma_*(I)=I$ as a sequence, the suspension map raises degree by $1$. The claim now follows from Proposition \ref{THHadm} and the formula (\ref{ast}) above. Indeed, only indecomposables, i.e.\ the generators of $\THH^{[n]}_*(H\mathbb{F}_p)$, survive in the colimit upon applying $\sigma_*$ (see Proposition \ref{kill}). Their indexing translates here, and the excess condition $pa_k<(p-1)(n+|I|)$, which restricts the order $[n]$, vanishes.
\end{proof}

Theorem \ref{DLR} appeared in \cite{DLR} as a remark, authors were mostly interested in the rings of integers, and their main result is 
\begin{theorem}[\cite{DLR}]\label{DLR_Z} For $n\geq1$ and odd $p$ $$\THH^{[n]}_*(H\mathbb{Z}; H\mathbb{F}_p)\simeq B^n(x_{2p})\otimes B^{n+1}(y_{2p-2})$$	
	where the iteration of bar construction starts from $P_{\mathbb{F}_p}(x_{2p})$ and $P_{\mathbb{F}_p}(y_{2p-2})$ respectively.\qed
\end{theorem}

\begin{proposition}\label{TAQint} The graded $\mathbb{F}_p$-module 	$$\TAQ_*(H\mathbb{Z},\mathbb{S}; H\mathbb{F}_p)$$ has a basis indexed by the sequences $I=(a_0, a_1,...,a_k)$ such that
	\begin{center}
		$a_0=0\ \mathrm{ or }\ 1;$
		
		$a_i=0\ \mathrm{ or }\ 1$ $(\mathrm{mod}\ 2p-2)\ \mathrm{for} \ 1\leq i\leq k;$
		
		$a_1\geq {2\left(p+a_0\right)}(p-1);$\quad $a_{i+1}\geq pa_i\ \mathrm{for} \ 1\leq i\leq k-1;$
	\end{center}
	of degrees $|I|=2p+a_0+a_1+...+a_k$.
\end{proposition}
\begin{proof}
	Let us denote $\nu:=x_{2p}$ and $\theta:=y_{2p-2}$. Then
		$$B^2\simeq E(\sigma\nu)\otimes\Gamma(\rho^0\sigma\theta)$$
	with degrees $2p+1$ and $2p$ respectively. Skipping an intermediate calculation for $\THH^{[n]}$, construct the bijection right away as follows. Both sets of admissible words $...\rho^0\sigma\nu$ and $...\rho^0\sigma\theta$ correspond to admissible sequences of the form (recall their construction in the Theorem \ref{Cartan}) 
	\begin{center}
		\begin{tabular}{c@{\hskip 50mm}c}
			$...\rho^0\sigma\nu$ & $...\rho^0\sigma\theta$ \\
			\multicolumn{2}{c}{$(a_1,...,a_k)$} \\
			$a_1\geq{\bf (2p+2)}(p-1)$ & $a_1\geq {\bf 2p}(p-1)$\\
			$|I|=1+a_1+...+a_k+{\bf 2p}$ & $ |I|=2+a_1+...+a_k+{\bf 2p-2}$
		\end{tabular}
	\end{center}
	
	Bold numbers appear here as in Proposition \ref{THHadm}. Those in the degree, since such are degrees of $\nu$ and $\theta$. The ones in the condition for $a_1$, since such are degrees of $\rho^0\sigma\nu$ and $\rho^0\sigma\theta$, and the minimal value of $a_1$ is attained when the next letter is $\gamma_p$. 
	
	Both sequences receive a shift by $1$ in degree (the numbers on the left in the sum) due to the induction step delay, see \ref{THHadm} again. The second sequence receives one more shift by $1$ akin to the Eilenberg-MacLane case, where an extra index $a_0$ is introduced to account for the divided power algebra in the base of induction. Mimicking that case, we introduce $a_0$, such that $a_0=1\leftrightarrow ...\nu$ and $a_0=0\leftrightarrow ...\theta$.
	
	The rest of the reasoning with stabilization is as in Proposition \ref{stabFp}, i.e.\ that we can drop the excess condition for the last index $a_k$. It is now straightforward to check, that the given indexing is a degree-preserving bijection.
\end{proof}

The best result we've got on more general coefficients is

\begin{proposition}\label{Zrational}\
	
	(i) $\TAQ_*(H\mathbb{Q},\mathbb{S})\simeq 0$.
	
	(ii) $\TAQ_*(H\mathbb{Z},\mathbb{S};H\mathbb{Q})\simeq 0$.
	
	(iii) The groups $\TAQ_*(H\mathbb{Z},\mathbb{S})$ are finite.
\end{proposition}
\begin{proof}
	(i) From the definition as a cotangent complex 
	$$\Omega_{H\mathbb{Q}|\mathbb{S}}\simeq QI(H\mathbb{Q}\wedge H\mathbb{Q}),$$
	one sees that since the multiplication map $\mu:H\mathbb{Q}\wedge H\mathbb{Q}\xrightarrow{\sim}H\mathbb{Q}$ is a homotopy equivalence, the augmentation ideal is trivial, therefore the cotangent complex itself is trivial. 
	
	(ii) Tansitivity cofiber sequence of $\mathbb{S}\to H\mathbb{Z}\to H\mathbb{Q}$ with the standard unit maps gives the long exact sequence
	$$...\to \TAQ_{*+1}(H\mathbb{Q},H\mathbb{Z})\to \TAQ_{*}(H\mathbb{Z},\mathbb{S};H\mathbb{Q})\to \TAQ_*(H\mathbb{Q},\mathbb{S})\to...$$
	here the terms on the right are trivial by claim (i). Terms on the left are trivial by the Lemma \ref{TAQ(Q,Z)}.
	
	(iii) Since $\Omega_{H\mathbb{Z}|\mathbb{S}}$ is an $H\mathbb{Z}$-module, it splits into summands in the same way as its homotopy groups. By the previous claim,  $$H\mathbb{Q}\wedge_{H\mathbb{Z}}\Omega_{H\mathbb{Z}|\mathbb{S}}\simeq \mathrm{pt.}$$ 
	So homotopy groups do not contain a free $\mathbb{Z}$-summands. Since they are finitely generated abelian groups, they are finite.
\end{proof}

\begin{remark*}
	It follows that the ranks (as abelian groups) of p-power torsion $T_n(p)=\TAQ_n(H\mathbb{Z},\mathbb{S})\otimes \mathbb{Z}_p$ could be read off inductively from $t_n(p)=\TAQ_n(H\mathbb{Z},\mathbb{S};H\mathbb{F}_p)$ as 
	$$\mathrm{rk}\ T_n(p)=
			\mathrm{rk}_{\mathbb{F}_p} t_n(p)-\mathrm{rk}\ T_{n-1}(p).
	$$
\end{remark*}

At last, we proceed to the case of $H\mathbb{Z}/p^m$. The identification of higher $\THH^{[n]}$ is due to
\begin{theorem}[see 9.1 in \cite{BHLKRZ}]
	$$\THH^{[n]}(H\mathbb{Z}/p^m; H\mathbb{F}_p)\simeq \THH^{[n]}(H\mathbb{Z};H\mathbb{F}_p)\underset{H\mathbb{F}_p}{\wedge} \mathrm{HH}^{[n]}(\mathbb{Z}/p^m; \mathbb{F}_p),$$
	where $\mathrm{HH}^{[n]}$ means the higher (algebraic) Hochschild homology, regarded as a simplicial $\mathbb{F}_p$-algebra. These are calculated as an iterated $\mathrm{Tor}$ functors
	$$\mathrm{HH}^{[n]}_*(\mathbb{Z}/p^m;\mathbb{F}_p)\simeq \mathrm{Tor}_*^{\mathrm{HH}^{[n-1]}_*(\mathbb{Z}/p^m; \mathbb{F}_p)}(\mathbb{F}_p, \mathbb{F}_p),$$ 
	provided that $\mathrm{HH}^{[1]}_*(\mathbb{Z}/p^m;\mathbb{F}_p)\simeq \Gamma_{\mathbb{F}_p}(\mu)$ with $\mathrm{deg}(\mu)=2$.
\end{theorem}

\begin{proposition}\label{taqzpn} The graded $\mathbb{F}_p$-module 	$$\TAQ_*(H\mathbb{Z}/p^m,\mathbb{S}; H\mathbb{F}_p)$$ is isomorphic to the direct sum of $\TAQ_*(H\mathbb{Z},\mathbb{S};H\mathbb{F}_p)$ and the module with basis indexed by the sequences $I=(b_1,...,b_k)$ such that
	\begin{center}
		
		$b_i=0\ \mathrm{ or }\ 1$ $(\mathrm{ mod }\ 2p-2)\ \mathrm{for} \ 1\leq i\leq k;$
		
		$b_1\geq 2(p-1);$\quad $b_{i+1}\geq pb_i\ \mathrm{for} \ 1\leq i\leq k-1;$
	\end{center}
	of degrees $|I|=1+b_1+...+b_k$.
\end{proposition}
\begin{proof}
A direct sum decomposition is implied by the presence of smash product of $H\mathbb{F}_p$-modules in the previous theorem. 

As for the second summand, which comes from the iterated $\mathrm{Tor}$ of $\Gamma(\mu)$, it has already appeared in the course of Eilenberg-MacLane calculations, and the argument is the same as in Theorem \ref{Cartan}. 
\end{proof}

\section{Richter spectral sequence}
\label{sGamma}
Recall the construction of ordinary Andr\'e-Quillen homology \cite{Quil}. $k$-derivation of commutative algebra $A$ with values in the module $M$ is a $k$-linear map $D\colon A\to M$, satisfying the Leibniz rule $D(ab)=aDb+bDa$. Denote the set of all derivations as $\textrm{Der}_k(A,M)$. Then the functor $\textrm{Der}_k(A,-)$ is representable by the $A$-module of relative K\"ahler differentials $\Omega_{A|k}$:
$$\textrm{Der}_k(A,M)=\mathrm{Hom}_A(\Omega_{A|k},M).$$

One can show that $\Omega_{A|k}=I/I^2$, where $I=\text{ker}(A\otimes_k A\xrightarrow{\mu}A)$ is the augmentation ideal. It is the quotient of the free $A$-module generated by $dx$, $x\in A$ modulo the relations 
$$d(x+y)=dx+dy,\qquad d(xy)=xdy+ydx,\qquad d(\alpha x)=\alpha dx\ \text{ for }\alpha\in k.$$

A square-zero extension $A\vee M\in \textrm{CAlg}_k/A$ is a commutative augmented $k$-algebra, coinciding with $A\oplus M$ as an abelian group, and with the ring structure $$(x,m)\cdot(y,m')=(xy,xm'+ym).$$

The trivial square-zero extension functor $M\to A\vee M$ is right adjoint to $\text{Ab}(X)=A\otimes_X\Omega_{X|k}$
$$\text{Ab}\colon \textrm{CAlg}_k/A\rightleftarrows \textrm{Mod}_A\colon A\vee -.$$

To define derived functors of nonadditive $\mathrm{Ab}$, one uses simplicial resolutions. The category $\Delta^{op}\textrm{CAlg}_k$ of simplicial commutative $k$-algebras possesses a model structure in which a map $f\colon X\to Y$ is a weak equivalence (resp. fibration) if it is a weak equivalence (fibration) of underlying simplicial sets. According to axioms, that already determines model structure, but it is possible to explicitly describe cofibrations too. This data descends to an augmented case, and an object in $\Delta^{op}\textrm{CAlg}_k/A$ is a diagram of the form $k_\bullet\to X\to A_\bullet$ with constant simplicial objects on the left and right hand side, the maps are resp. degreewise unit and augmentation. The desired simplicial resolution of $A$ is an object $P_\bullet\in \Delta^{op}\textrm{CAlg}_k/A$ representing its cofibrant replacement.

\begin{definition}
	\textit{Cotangent complex} of $A\in\textrm{CAlg}_k/A$ is a simplicial $A$-module
	$$L_{A|k}:=\textbf{L}\mathrm{Ab}(A)=A\otimes_{P_\bullet}\Omega_{P_\bullet|k}$$
	\textit{Andr\'e-Quillen homology} of $A$ with coefficients in a module $M$ is 
	$$\AQ_*(A|k;M):=\pi_*(L_{A|k}\otimes_A M).$$
\end{definition}

By the Dold-Kan correspondence, the category of simplicial $A$-modules is equivalent to the category of bounded below chain complexes of $A$-modules. Homotopy groups correspond to the homology of normalized complex. 

\begin{theorem}\label{AQpropert}
	Cotangent complex enjoys the following properties
	\begin{itemize}
		\item (Localization) For $S$ a multiplicative subset in $A$, we have $L_{S^{-1}A|A}\simeq0$.
		\item (Smooth vanishing) For $A\to B$ a smooth morphism, $L_{B|A}\simeq\Omega_{B|A}$.\qed
	\end{itemize}
\end{theorem}

\begin{lemma}[see 6.13 in \cite{Quil}]\label{AQFp} Let $B=A/I$. Then for $I$ a regular ideal, $L_{B|A}\simeq I/I^2[1]$.
\end{lemma}

Therefore $\AQ_*(\mathbb{Z}/p^n|\mathbb{Z})\simeq \mathbb{Z}/p^n[1]$, but note, that the Quillen's theorem is much more general, and for this case it is straightforward to perform a simplicial resolution calculation.

In \cite{Rich} Richter constructs an Atiyah-Hirzebruch type spectral sequence, which takes algebraic Andr\'e-Quillen homology as input and converges to topological one. The role of the basepoint there is due to the polynomial ring in one variable. Let $A$ be an arbitrary commutative augmented $k$-algebra, then spectral sequence takes the form
$$E^2_{p,q}=\TAQ_q\big(Hk[x], Hk; Hk\big)\otimes \AQ_p(A|k;k)\Rightarrow \TAQ_{p+q}(HA, Hk; Hk)$$

Other deep theorem of Richter says that $\TAQ(Hk[x],Hk)\simeq H\mathbb{Z}\wedge Hk$. This concludes the identification of the second page terms. As an immediate consequence 

\begin{lemma}\label{TAQ(Q,Z)}$\TAQ_*(H\mathbb{Q},H\mathbb{Z})\simeq 0$.
\end{lemma}
\begin{proof}
	The algebra $\mathbb{Q}|\mathbb{Z}$ is a localization of $\mathbb{Z}$. By properties \ref{AQpropert}, its homology are trivial ${\AQ_*(\mathbb{Q}|\mathbb{Z})\simeq 0}$. Conclude the triviality of $\TAQ_*$ by a spectral sequence argument.	
\end{proof}

In order to work over the $H\mathbb{Z}$ base in nontrivial cases, we need to make precise a folklore statement on the integral Steenrod algebra
\begin{proposition}\label{intSt}
	There is an abelian groups isomorphism $H\mathbb{Z}^*H\mathbb{Z}\simeq \sideset{}{_p}{\bigoplus} \mathcal{A}_p/\beta$ with a copy of $\mathbb{Z}$ in degree zero. Here $\mathcal{A}_p$ is the $\mathrm{mod\ p}$ Steenrod algebra, the quotient is taken wrt the two-sided ideal generated by the Bockstein operation $\beta$.
\end{proposition}
\begin{proof}
	According to Serre, $H\mathbb{F}_p^*H\mathbb{Z}\simeq \mathcal{A}_p/\beta$. For $p=2$, the Bockstein $\beta$ coincides with the first Steenrod square $Sq^1$. We will run the Adams spectral sequence
	$$\mathrm{Ext}_{\mathcal{A}_p}^{s,t}(\mathcal{A}_p, \mathcal{A}_p/\beta)\Rightarrow H\mathbb{Z}^{t-s}H\mathbb{Z}\otimes \mathbb{Z}_p$$
	Further, use the change of rings isomorpism, or a free resolution to compute Ext
	$$...\xrightarrow{\beta} \mathcal{A}_p\xrightarrow{\beta}\mathcal{A}_p\to \mathcal{A}_p/\beta$$
	Applying $\textrm{Hom}_{\mathcal{A}_p}(-,\Sigma^*\mathcal{A}_p/\beta)$, obtain complex 
	$$\Sigma^*\mathcal{A}_p/\beta\xrightarrow{\beta}\Sigma^*\mathcal{A}_p/\beta\xrightarrow{\beta}...$$
	which is exact everywhere except for $s=0$ with homology $\mathcal{A}/\beta$ in that degree. Thus, the Adams spectral sequence degenerates and $H\mathbb{Z}^*H\mathbb{Z}_p\simeq \mathcal{A}_p/\beta$.
	
	Next, apply $\pi_*\textrm{Map}(H\mathbb{Z},-)$ to a Bousfield square
	\begin{align*}
		\xymatrix{
			H\mathbb{Z} \ar[d] \ar[r] & \prod_p H\mathbb{Z}_p \ar[d]\\
			H\mathbb{Q}  \ar[r] &  \prod_p H\mathbb{Q}_p
		}
	\end{align*}
	to get the claimed isomorphism.
\end{proof}

\begin{proposition}\label{TAQ(Z/p^n,Z)}
	$$\TAQ_*(H\mathbb{Z}/p^n,H\mathbb{Z})\simeq\Sigma^2\mathcal{A}_p/\beta$$ is an isomorphism of graded $\mathbb{F}_p$-modules. There is a copy of $\mathbb{F}_p$ in degree 1, not 2.
\end{proposition}
\begin{proof}
	By the universal coefficients theorem, the torsion in homology shifts by $-1$ upon passing to cohomology 
	\begin{center}
		$0\to\text{Ext}^1_\mathbb{Z}(H_{i-1},\mathbb{Z})\to H^i\to \mathrm{Hom}_\mathbb{Z}(H_i,\mathbb{Z})\to 0$
		
		\vspace{2mm}
		$ H^i\simeq \mathrm{Free}(H_i)\oplus \mathrm{Tors}(H_{i-1})$
	\end{center}
	
	From previous Proposition \ref{intSt}, cohomology $H\mathbb{Z}^*H\mathbb{Z}$ is all torsion except for degree 0. Therefore homology $H\mathbb{Z}_*H\mathbb{Z}\simeq\bigoplus\Sigma\mathcal{A}_p/\beta$ shifts by one with the indicated caveats.
	
	Then we have $\AQ_*(\mathbb{Z}/p^n|\mathbb{Z})\simeq \Sigma\mathbb{Z}/p^n$ by Proposition \ref{AQFp}. Richter spectral sequence is concentrated in the single first row, thus degenerates, and
	$$
	H\mathbb{Z}_*H\mathbb{Z}\otimes_\mathbb{Z}\Sigma\mathbb{Z}/p^n\simeq\Sigma^2\mathcal{A}_p/\beta$$
	where, due to the remarks, we have $\mathbb{F}_p$ in degree $1$, and not $2$.
\end{proof}

\begin{remark*}
Similarly, $\TAQ_*$ of any smooth $\mathbb{Z}$-algebra is the tensor product of its K\"ahler differentials with the dual Steenrod algebra.
\end{remark*}

\section{Spectral Lie algebras}
\label{sLie}

Let us construct a two-sided bar construction for the functor $\mathbb{P}$ of the free commutative algebra and the spectrum $X\in\text{CAlg}_{\mathbb{S}}/H\mathbb{F}_p$, i.e.\ an augmented simplicial object in $\text{CAlg}_{\mathbb{S}}/H\mathbb{F}_p$ of the form
\begin{align*}
	\xymatrix{
		\big|...\ \mathbb{P}^3X \ar@<0ex>[r] \ar@<1ex>[r] \ar@<-1ex>[r] & \ar@<-0.5ex>[r] \ar@<0.5ex>[r]\mathbb{P}^2X & \mathbb{P}X\big| \ar@<0ex>[r]^{\sim} & X
	}
\end{align*}
Its geometric realization is equivalent to $X$ as a module. Strictly speaking, we should compose $\mathbb{P}$ with a forgetful functor everywhere, but this messes up notations. Apply a colimits preserving functor $\TAQ(-,\mathbb{S};H\mathbb{F}_p)$ to that diagram. Then $\TAQ(\mathbb{P}^nX,\mathbb{S};H\mathbb{F}_p)\simeq H\mathbb{F}_p\wedge \mathbb{P}^{n-1}X$ by the Lemma \ref{taqfree} on free algebras, and there is an equivalence
$$ H\mathbb{F}_p\wedge\left| \text{B}_\bullet(\textrm{id}, \mathbb{P}, X)\right|\simeq \TAQ(X,\mathbb{S};H\mathbb{F}_p).$$

The expression on the left should be regarded as a free coalgebra over the Koszul-dual Lie cooperad. Now, consider the adjunction between the cotangent complex $\Omega_{(-)|H\mathbb{F}_p}$ and the trivial square-zero extension $H\mathbb{F}_p\vee(-)$. Their composition defines an endofunctor on modules
$$\text{Lie}^{\pi}_{\mathbb{F}_p}(-):= \left| \text{B}_\bullet\left(\textrm{id}, \mathbb{P}, H\mathbb{F}_p\vee (-)^{*}\right)\right|^*$$
here $(-)^{*}$ denotes the dual module, i.e.\ $M^{*}\simeq \text{Map}_{H\mathbb{F}_p}(H\mathbb{F}_p,M)$. Brantner and Mathew provide an explicit description of the homotopy groups $\pi_*\text{Lie}^{\pi}_{\mathbb{F}_p}(M)$ for a finitely generated module $M$. By representability this gives all $m$-ary $\TAQ$-cohomological operations. We need the following 

\begin{theorem}[see 7.7 in \cite{BrMat}]\label{BMLie}
	 Let $\ell$ be an odd integer if $p$ is odd. Then $\pi_{\ast}\big(\mathrm{Lie}^{\pi}_{\mathbb{F}_p}(\Sigma^\ell\mathbb{F}_p)\big)$ has a basis indexed by sequences $(i_1,...,i_k)$ s.t.
	 \begin{itemize}
	 	\item Each $i_j$ equals $0$ or $1$ modulo $2(p-1)$.
	 	\item For all $1\leq j<k$ we have $i_j<pi_{j+1}$.
	 	\item We have $i_k\leq (p-1)\ell$. 
	 \end{itemize}
	Degree of an element corresponding to $(i_1,...,i_k)$ equals $\ell+i_1+...+i_k-k$.\qed
\end{theorem}

The following calculation is originally due to Basterra and Mandell, but the harder second part was obtained, presumably, by the method, different from that in Theorem \ref{BMLie}. In particular, the concept of a spectral Lie algebra was not yet introduced.

\begin{proposition}\label{LieFp}
	$\TAQ_*(H\mathbb{F}_p,\mathbb{S})\simeq\pi_{-2-*}\big(\mathrm{Lie}^{\pi}_{\mathbb{F}_p}(\Sigma^{-3}\mathbb{F}_p)\big).$ 
\end{proposition}
\vspace{\parsep}
\noindent\textit{Proof.} Let's start with some transformations, as in \cite{Laz}. First, use the basechange to $H\mathbb{F}_p$ and the definition via stabilization 
	\begin{align*}
		\Omega_{H\mathbb{F}_p|{\mathbb{S}}}\simeq\Omega_{H\mathbb{F}_p\wedge H\mathbb{F}_p|H\mathbb{F}_p}\simeq&\  \underset{n}{\textrm{hocolim}}\ \Sigma^{-n}IE^n(H\mathbb{F}_p\wedge H\mathbb{F}_p)\\
		\simeq&\ \Sigma^{-2}\underset{n}{\textrm{hocolim}}\ \Sigma^{-n+2}IE^{n-2}\big(E^2\{H\mathbb{F}_p\wedge H\mathbb{F}_p\}\big)
	\end{align*}
	In other words, up to a shift, we can start stabilization from any member of the spectrum. The algebra $E^2\{H\mathbb{F}_p\wedge H\mathbb{F}_p\}$ has homotopy groups $\THH^{[2]}_*(H\mathbb{F}_p)\simeq E(\sigma\mu)$ by the degeneration of the Eilenberg-Moore spectral sequence
	$$\text{Tor}^{\THH_\ast(H\mathbb{F}_p)}_{*,*}(\mathbb{F}_p,\mathbb{F}_p)\Rightarrow \THH^{[2]}_{*}(H\mathbb{F}_p).$$
	And from the Mandell's Lemma \ref{ManLemma} we know that a commutative $H\mathbb{F}_p$-algebra with such homotopy groups is equivalent to the split square-zero extension $H\mathbb{F}_p\vee \Sigma^3 H\mathbb{F}_p$. Thus
	$$\TAQ(H\mathbb{F}_p,{\mathbb{S}})\simeq \Sigma^{-2}\TAQ(H\mathbb{F}_p\vee \Sigma^3 H\mathbb{F}_p,H\mathbb{F}_p;H\mathbb{F}_p).$$
	Due to Theorem \ref{BMLie}, now it is possible to write down the claimed isomorphism of graded $\mathbb{F}_p$-modules
	\[
	\pushQED{\qed} 
	\TAQ_*(H\mathbb{F}_p,\mathbb{S})\simeq\pi_{-2-*}\big(\mathrm{Lie}^{\pi}_{\mathbb{F}_p}(\Sigma^{-3}\mathbb{F}_p)\big).\qedhere
	\popQED
	\]
\vspace{\parsep}
After some careful manipulations with the minuses signs, one can see that one of the resulting inequalities differs from our Proposition \ref{stabFp}, where it has $4$ instead of $3$ from the last bullet in the Theorem \ref{BMLie}. However, this distinction does not affect the set of indices, except for $p=2$, when both theorems require modifications. Our proof undergoes minor changes, related to the behaviour of divided power algebra in even characteristic. Alternatively, one could look up the answer for that case in \cite{Laz}.

{\footnotesize
	\emph{Affiliation}:
		National Research University Higher School of Economics.}
		

\begin{thebibliography}{}
	\bibitem{B1} Basterra {\bf Andr\'{e}-Quillen cohomology of commutative S-algebras} J. Pure Appl. Algebra 144 (1999), no. 2, 111–143.
	\bibitem{BM} Basterra, Mandell {\bf Homology and cohomology of $E_\infty$ ring spectra}  Math. Z. 249 (2005), no. 4, 903–944.
	\bibitem{BHLKRZ} Bobkova, Höning, Lindenstrauss, Poirier, Richter, Zakharevich {\bf Splittings and calculational techniques for higher THH} Algebr. Geom. Topol. 19 (7) 3711 - 3753, 2019.
	\bibitem{BrMat} Brantner, Mathew {\bf Deformation Theory and Partition Lie Algebras} Acta Math., 235 (2025), 1–148.
	\bibitem{Car} Cartan {\bf Détermination des algèbres $H_*(\pi , n; Z_p)$ et $H^*(\pi , n ;Z_p)$, $p$ premier impair} Séminaire Henri Cartan, Vol. 7 (1954-1955) no. 1, Talk no. 9, 10 p.
	\bibitem{DLR} Dundas, Lindenstrauss, Richter {\bf On higher topological Hochschild homology of rings of integers} Mathematical Research Letters 25, 2 (2018), 489–507.
	\bibitem{EKMM} Elmendorf, Kriz, Mandell, May {\bf Rings, modules and algebras in stable homotopy theory} Mathematical Surveys and Monographs,
	47. American Mathematical Society, Providence, RI, (1997), xii+249.
	\bibitem{Kld} Kaledin {\bf Trace theories, Bokstedt periodicity and Bott periodicity} {\tt arxiv:2004.04279} 
	\bibitem{Kit} Kitchloo {\bf H(Z/$p^k$) as a Thom spectrum and topological Hochschild homology} Proc. Amer. Math. Soc., 148(8), 3647-3651.
	\bibitem{Laz} Lazarev {\bf Cohomology theories for highly structured ring spectra} In: Structured ring spectra, London Math. Soc. Lecture Note Ser., 315, Cambridge Univ. Press, Cambridge, (2004), 201–231
	\bibitem{Quil} Quillen {\bf On the (co-) homology of commutative rings} (Proc. Sympos. Pure Math., Vol. XVII, New York, 1968) Amer. Math. Soc., Providence, R.I., 1970, pp. 65–87. 	
	\bibitem{Rich} Richter {\bf An Atiyah-Hirzebruch spectral sequence for topological Andr\'{e}-Quillen homology} J. Pure Appl. Algebra 171 (2002), no. 1, 59–66.
	\bibitem{Veen} Veen {\bf Detecting periodic elements in higher topological Hochschild homology} Geom. Topol. 22 (2018) no. 2, 693–756.
\end{thebibliography}
\end{document}